\documentclass[11pt]{article}

\usepackage[pdftex]{graphicx}

\usepackage{amssymb}
\usepackage{amsthm}

\usepackage{accents}

\usepackage{tipa}

\usepackage[title]{appendix}

\makeatletter
\renewcommand\section{\@startsection{section}{1}{\z@}%
                                  {-3.5ex \@plus -1ex \@minus -.2ex}%
                                  {2.3ex \@plus.2ex}%
                                  {\normalfont\large\bfseries}}
\makeatother

\DeclareGraphicsExtensions{.jpg}
\begin{document}

\title{Domination structure in 3-connected graphs}

\author{Misa Nakanishi \thanks{E-mail address : nakanishi@2004.jukuin.keio.ac.jp}}
\date{}
\maketitle

\begin{abstract}
From a recent perspective, the structure of a 3-connected graph is studied in this paper. It stipulates the minimum dominating set of a 3-connected graph. Also, we count the number of structures, as a consequence, the upper bound is obtained. By it, the minimum dominating set of a 3-connected graph is determined in polynomial time.\\
keywords: 3-connected graph, minimum dominating set, ear decomposition
\end{abstract}

\newtheorem{thm}{Theorem}[section]
\newtheorem{lem}{Lemma}[section]
\newtheorem{prop}{Proposition}[section]
\newtheorem{cor}{Corollary}[section]
\newtheorem{rem}{Remark}[section]
\newtheorem{conj}{Conjecture}[section]
\newtheorem{claim}{Claim}[section]
\newtheorem{obs}{Observation}[section]
\newtheorem{fact}{Fact}[section]

\newtheorem{defn}{Definition}[section]
\newtheorem{propa}{Proposition}
\renewcommand{\thepropa}{\Alph{propa}}
\newtheorem{conja}[propa]{Conjecture}

\section{Introduction}
\label{intro}

\noindent In this paper, a graph $G$ is finite, undirected, and simple with the vertex set $V$ and edge set $E$. We follow \cite{Diestel} for basic notations and properties. The open neighborhood of a vertex $v \in V(G)$, denoted by $N_G(v)$, is $\{ u \in V(G)\colon\ uv \in E(G) \}$, and the closed neighborhood of a vertex $v \in V(G)$, denoted by $N_G[v]$, is $N_G(v) \cup \{v\}$. For a set $W \subseteq V(G)$, let $N_G(W) = \bigcup_{v \in W} N_G(v)$ and $N_G[W] = N_G(W) \cup W$. A {\it dominating
set} $X \subseteq V(G)$ is such that $N_G[X] = V(G)$. A minimum dominating set is called a {\it d-set}. \\

It is fundamental to cover a graph by paths of length 0 mod 3 so that a dominating set is stipulated. For cubic graph domination by Reed \cite{Reed}, a cubic graph was the object to be covered by paths of length 0 mod 3. A dominating set is given by taking every three vertices on the paths. 
To minimize a dominating set, it is important that the paths are connecting appropriately.
By focusing on the connection of these paths, we observed that if any two cycles of length 0 mod 3 have exactly one common path to intersect with then it is optimum for minimum domination, so we called the graph of these cycles {\it structure}. A 3-connected graph is almost covered by these cycles. As for a 2-connected graph, the next property is familiar to us.  

\begin{propa}[\cite{Diestel}]
A graph is 2-connected if and only if it can be constructed from a cycle by successively adding $H$-paths to graphs $H$ already constructed.
\end{propa}

In this paper, we reveal that for a 3-connected graph, each component not covered by certain structure is trivial, say an exceptional vertex. When we assign a label every three vertices of the cycles, we find that an exceptional vertex is all adjacent to labeled vertices. That is to say, a 3-connected graph is explained only by this structure to have a minimum dominating set. Especially, we call the structure that attains a d-set the {\it domination structure}.\\

Each structure is constructed by finding cycles of length 0 mod 3 one by one with one path intersection. That is to say, it is polynomially determined. In this paper, the number of structures we can take in a 3-connected graph is considered. By revealing that the upper bound is the order plus one, we know that the domination structure is determined in polynomial time. Therefore, the minimum dominating set of a 3-connected graph is determined in polynomial time although this problem is NP-complete.

\section{Preliminary}

For the proofs of our theorem, we introduce terminology and notations. Let $C_G$ be the collection of all cycles with length 0 mod 3 in a graph $G$.
Two cycles $C_1$ and $C_2$ are connecting without {\it seam} if $C_1 \cap C_2$ is one path. For two cycles $C_1$ and $C_2$ such that $C_1 \cap C_2$ is $k$ paths, we say $C_2$ is obtained by adding $k$ $C_1$-paths ($k \geq 1$).
Let $C_{SG}$ be a maximal
subset of $C_G$ such that $|C_{SG}| = 1$, or for all $C_1 \in C_{SG}$, there exists $C_2 \in C_{SG}$ connecting without seam with $C_1$. Let $\texttoptiebar{C}_{SG}$ be the graph $\bigcup_{C \in C_{SG}} C$. Let $D_{SG}$ be a maximal subset of $C_{SG}$ such that a cycle of $C_{SG}$ is dropped when no exclusive vertex is contained in it. 
A {\it $X$-3-path} is a path that has a vertex of $X \subseteq V(G)$ at every three vertices in the sequence. 
That $X$-3-paths are {\it assigned} to $\texttoptiebar{C}_{SG}$ means for $\texttoptiebar{C}_{SG}$, to define $X \subseteq V(\texttoptiebar{C}_{SG})$ so that every cycle of $D_{SG} \subseteq C_{SG}$ has a vertex of $X$ at every three vertices in the sequence.  
Let $\mathcal{C}(G)$ be the set of all $C_{SG}$. 
Let $\mathcal{C}_t(G)$ be the set of all $\texttoptiebar{C}_{SG}$. 
Let $\mathcal{F}(G)$ be the set of all maximal unions of disjoint members in $\mathcal{C}_t(G)$.

\section{Domination structure in 3-connected graphs}
\label{sec:3} 

\begin{thm}\label{T2}
For a 3-connected graph $G$, there exists $H \in \mathcal{F}(G)$ such that every component $R$ of $G - H$ is $|R| \leq 1$. Moreover, if $X$-3-paths are assigned to the component of $H$ then for $R = \{x\}$, $N_G(x) \subseteq X$.
\end{thm}

\begin{proof}
Let $G$ be a 3-connected graph. Let $R$ be a connected subgraph of $G$. 

\begin{claim}\label{R}
For some distinct three vertices $t_0, u_0, v_0 \in N_G(V(G - R)) \setminus V(G - R)$ and $d_R(t_0) \geq 2$, there exist a path from $t_0$ to $u_0$ and a path from $t_0$ to $v_0$ that are internally disjoint in $R$. 
\end{claim}

\begin{proof} 
Let $t_1, t_2 \in N_G(t_0) \cap R$. Let $P$ be a path between $t_1$ and $u_0$ in $R$, and $Q$ be a path between $t_2$ and $v_0$ in $R$. Suppose $P \cap Q \ne \emptyset$ for such $P$ and $Q$. Let $x \in P \cap Q$ be the first vertex of $P$ from $t_1$. Since $G$ is 3-connected, $G - x - t_0$ is connected. That is, without loss of generality, there exists a path that avoids $x$ from a vertex of $P\ring{x}$ to some $w_0 \in N_G(V(G - R)) \setminus V(G - R)$ distinct from $t_0$ and $v_0$. Therefore, there exists a path $P'$ between $t_1$ and $w_0$ in $R$ such that $P' \cap Q = \emptyset$ as required.
\end{proof}

\begin{claim}\label{emptyset}
$\mathcal{C}(G) \ne \emptyset$.
\end{claim}

\begin{proof}
Take any vertex $x \in V(G)$. Since $G$ is 3-connected, $x$ is adjacent to at least 3 vertices. Also, for all $y \in N_G(x)$, it follows that $d_{G - x}(y) \geq 2$. For some $t_0, u_0, v_0 \in N_G(x)$, by Claim \ref{R}, there exist a path $P_1$ from $t_0$ to $u_0$, a path $P_2$ from $u_0$ to $v_0$, and a path $P_3$ from $v_0$ to $t_0$ that are internally disjoint in $G - x$. It suffices to show that there exists a path $Q$ of length 1 mod 3 with two ends of $\{t_0, u_0, v_0\}$. If $P_1$, $P_2$, or $P_3$ is of length 1 mod 3, then the proof is complete. If at least two of $P_1$, $P_2$, and $P_3$ are of length 2 mod 3, assume $||P_1|| \equiv 2$ and $||P_2|| \equiv 2$ mod 3. Now, $||P_1u_0P_2|| \equiv 1$ mod 3. Otherwise, without loss of generality, two cases arise. 
\begin{itemize}
\item[(I)] $||P_1|| \equiv 2$, $||P_2|| \equiv 0$, and $||P_3|| \equiv 0$ mod 3. 
\item[(II)] $||P_1|| \equiv 0$, $||P_2|| \equiv 0$, and $||P_3|| \equiv 0$ mod 3. 
\end{itemize}
Note that the inner vertices of $P_1$, $P_2$, and $P_3$ have paths between them in $G - x$, otherwise adjacent to $x$. By simple case analysis as ($*0$), we obtain $Q$ through the paths, see ($*T1$). 
\end{proof}

\noindent ($*0$) \\
Let $\mathcal{S}_i$ be the set of all paths that have an inner vertex $i$ as an end ($1 \leq i \leq 9$). \\
Step 1. Let $i = 1$. \\
Step 2. Confirm that we obtain $Q$ that contains a path $S \in \mathcal{S}_i$. \\
Step 3. If all $S \in \mathcal{S}_i$ have $Q$, stop the steps. Otherwise let $\mathcal{T}_i$ be the set of all $S \in \mathcal{S}_i$ that does not have $Q$. \\
Step 4. Increment $i$. \\
Step 5. For $\mathcal{S}_i$, apply Step 2 and Step 3, and go to Step 6. \\
Step 6. Confirm that we obtain $Q$ that contains some paths in a member $T \in \mathcal{T}_1 \times \mathcal{T}_2 \times \cdots \times \mathcal{T}_i$. \\
Step 7. If all $T \in \mathcal{T}_1 \times \mathcal{T}_2 \times \cdots \times \mathcal{T}_i$ have $Q$, stop the steps. Otherwise go to Step 4.\\

\begin{claim}\label{p1}
For a d-set $Y$ of $G$, there exists $\texttoptiebar{C}_{SG} \in \mathcal{C}_t(G)$ such that $X$-3-paths can be assigned to $\texttoptiebar{C}_{SG}$ so that $X \subseteq Y$. 
\end{claim}

\begin{proof}
It is obvious that for a d-set $Y$ of $G$, there exists $\texttoptiebar{C}_{SG} \in \mathcal{C}_t(G)$ such that every $D \in D_{SG} \subseteq C_{SG}$ has a vertex of $Y$ at every three vertices in the sequence. Indeed, if a member $C \in C_{SG}$ does not have a vertex of $Y$ at every three vertices in the sequence, then $C \cap Y = \emptyset$, but every $D \in D_{SG}$ has at least one exclusive vertex, so $|D \cap Y| = |D|/3$.
\end{proof}

\begin{claim}\label{closed path}
For a member $\texttoptiebar{C}_{SG} \in \mathcal{C}_t(G)$, if $X$-3-paths are assigned to $\texttoptiebar{C}_{SG}$ then any two vertices of $\texttoptiebar{C}_{SG}$ have at least two paths between them that have distinct penultimate vertices from both ends in $\texttoptiebar{C}_{SG}$, and that a $X$-3-path is assigned to. (We call it a closed $X$-3-path.)
\end{claim}

\begin{proof}
For two cycles of $C_{SG}$ that connect without seam, say $C_1$ and $C_2$, let $P_1 = C_2 - C_1$. Set $P_1 = v_1 \cdots v_2$. Let $x_1 \in V(C_1)$ and $x_2 \in V(P_1)$. Now,  there exist two paths $x_1C_1v_1P_1x_2$ and $x_1C_1v_2P_1x_2$ that a $X$-3-path is  assigned to. Let $x_1 \in V(C_1)$ and $x_2 \in V(C_2 - P_1)$. Now, there exist two paths $x_1C_1v_1C_1x_2$ and $x_1C_1v_2C_1x_2$ that a $X$-3-path is assigned to. For two cycles $C_k$ and $C_{k + 1}$ of $C_{SG}$ that connect without seam ($1 \leq k \leq p$, $p \geq 1$), by applying the same argument, a vertex $x_k \in V(C_k)$ and a vertex $x_{k + 1} \in V(C_{k + 1})$ have two paths between them that are internally disjoint in $\texttoptiebar{C}_{SG}$, and that a $X$-3-path is assigned to. Any two vertices $x_i, x_j \in \texttoptiebar{C}_{SG}$ are contained in some $C_i, C_j \in C_{SG}$, $i < j$ respectively. By the definition of $C_{SG}$, without loss of generality, $C_i, C_j \in \{C_k\colon\ 1 \leq k \leq p + 1\}$. That is, for some $x_{i + 1} \in V(C_{i + 1}), \cdots, x_{j - 1} \in V(C_{j - 1})$, $x_i$ and $x_j$ have two paths between them that have distinct penultimate vertices from both ends through $x_{i + 1}, \cdots, x_{j - 1}$ in $\texttoptiebar{C}_{SG}$, and that a $X$-3-path is assigned to. 
\end{proof}

Now, for a member $\texttoptiebar{C}_{SG} \in \mathcal{C}_t(G)$, let $R$ be a component of $G - \texttoptiebar{C}_{SG}$. Note that $N_G(R) \cap \texttoptiebar{C}_{SG}$ has at least 3 vertices since $G$ is 3-connected. Let $U = N_G(R) \cap \texttoptiebar{C}_{SG}$. 
Let $X$-3-paths be assigned to $\texttoptiebar{C}_{SG}$. A vertex $u \in \texttoptiebar{C}_{SG}$ is two types, $u \in X$ or $u \in V(G) \setminus X$. 
For a vertex $r \in R$ such that $N_G(r) \subseteq \{s\} \cup U$ for some $s \in R$, let $R'$ be the set of all $r$. Let $O = \texttoptiebar{C}_{SG} \cup R'$. Let $M$ be a component of $G - O$. Now, take $t_0, u_0, v_0 \in N_G(O) \cap M$ arbitrarily. The cycle obtained from a path from $t_0$ to $u_0$, a path from $u_0$ to $v_0$, and a path from $v_0$ to $t_0$ that are internally disjoint in $M$ is denoted by $\blacktriangle$ if there exists. Let $t, u, v \in O$ be adjacent to $t_0$, $u_0$, and $v_0$ respectively. A vertex $o \in \{t, u, v\}$ is four types ($*1$). 
\begin{itemize}
\item[(a)] $o \in \texttoptiebar{C}_{SG}$ and $o \in X$ 
\item[(b)] $o \in R'$ and $(N_O(o) \cap \texttoptiebar{C}_{SG}) \cap X = \emptyset$ 
\item[(c)] $o \in R'$ and $(N_O(o) \cap \texttoptiebar{C}_{SG}) \cap X \ne \emptyset$ 
\item[(d)] $o \in \texttoptiebar{C}_{SG}$ and $o \in V(G) \setminus X$ 
\end{itemize}
Note that for (c), $(N_O(o) \cap \texttoptiebar{C}_{SG}) \subseteq X$, otherwise $o \in \texttoptiebar{C}_{SG}$. 
If there exists a path $Q$ between two vertices of $\{t, u, v\}$ (through the vertices in $M$) of specific length and types, then $C_{SG}$ is not maximal. Twenty cases arise ($*2$) by simple case analysis. For example, a path $Q$ of length 2 mod 3 between types (a) and (a) is the case. \\

\begin{table}[h]
\noindent ($*2$)
  \begin{minipage}[t]{.45\textwidth}
    \begin{center}
      \begin{tabular}{|c|c|c|} \hline
    \ & length & types \\ \hline
(1) & 2 mod 3 & (a) and (a) \\ \hline
(2) & 0 mod 3 & (a) and (b) \\ \hline
(3) & 2 mod 3 & (a) and (b) \\ \hline
(4) & 1 mod 3 & (a) and (c) \\ \hline
(5) & 0 mod 3 & (a) and (d) \\ \hline
(6) & 1 mod 3 & (a) and (d) \\ \hline
(7) & 0 mod 3 & (b) and (b) \\ \hline
(8) & 1 mod 3 & (b) and (b) \\ \hline
(9) & 2 mod 3 & (b) and (b) \\ \hline
(10) & 1 mod 3 & (b) and (c) \\ \hline

      \end{tabular}
    \end{center}

  \end{minipage}
  \hfill
  \begin{minipage}[t]{.45\textwidth}
    \begin{center}
      \begin{tabular}{|c|c|c|} \hline
    \ & length & types \\ \hline
(11) & 2 mod 3 & (b) and (c) \\ \hline
(12) & 0 mod 3 & (b) and (d) \\ \hline
(13) & 1 mod 3 & (b) and (d) \\ \hline
(14) & 2 mod 3 & (b) and (d) \\ \hline
(15) & 0 mod 3 & (c) and (c) \\ \hline
(16) & 0 mod 3 & (c) and (d) \\ \hline
(17) & 2 mod 3 & (c) and (d) \\ \hline
(18) & 0 mod 3 & (d) and (d) \\ \hline
(19) & 1 mod 3 & (d) and (d) \\ \hline
(20) & 2 mod 3 & (d) and (d) \\ \hline

      \end{tabular}
    \end{center}

  \end{minipage}
\end{table}

\noindent Note that, by Claim \ref{closed path}, two vertices in ($*1$) have a $X$-3-path in $\texttoptiebar{C}_{SG}$ but the same pair may have a different connection, which appears doubly in ($*2$). 

\begin{claim}\label{triangle}
If $M$ has $\blacktriangle$, then $M = \emptyset$.
\end{claim}

\begin{proof}
Suppose $M \ne \emptyset$. To the contrary, it suffices to show that there exists $Q$ as in ($*2$). Let $P_1$ be a path from $t_0$ to $u_0$, $P_2$ be a path from $u_0$ to $v_0$, and $P_3$ be a path from $v_0$ to $t_0$ that are internally disjoint in $M$. According to the types ($*1$) of $t, u$, and $v$, if $||P_1||$, $||P_2||$, or $||P_3||$ is specified as ($*2$) then we obtain $Q$. Suppose not. Note that the inner vertices of $P_1$, $P_2$, and $P_3$ have paths between them in $M$, otherwise adjacent to the vertex of $O$. By simple case analysis as ($*0$), we obtain $Q$ through the paths, see ($*T2$). 
\end{proof}

\begin{claim}\label{R1}
For all $H \in \mathcal{F}(G)$, $H$ has at most one component. Moreover, every component of $G - H$ is a tree.
\end{claim}

\begin{proof}
It is straightforward from Claim \ref{triangle}.
\end{proof}

By Claim \ref{p1} and Claim \ref{R1}, for a d-set $Y$ of $G$, there exists $H \in \mathcal{F}(G)$ such that $X$-3-paths can be assigned to the  component of $H$ so that $X \subseteq Y$.
Now, let $R$ be a component of $G - H$.

\begin{claim}\label{LC}
$|R| \leq 1$, and if $|R| = 1$ then $x \in R$ satisfies $N_G(x) \subseteq X$.
\end{claim}

\begin{proof}
By Claim \ref{R1}, it suffices to consider that $|R| \leq 2$, and $H$ equals some $\texttoptiebar{C}_{SG} \in \mathcal{C}_t(G)$ and $N_G(R) \setminus R \subseteq \texttoptiebar{C}_{SG}$. Suppose $R = \{u, v\}$. Case (i) $(N_G(u) \cap \texttoptiebar{C}_{SG}) \cap X \ne \emptyset$ and $(N_G(v) \cap \texttoptiebar{C}_{SG}) \cap X \ne \emptyset$. Let $u' \in (N_G(u) \cap \texttoptiebar{C}_{SG}) \cap X$ and $v' \in (N_G(v) \cap \texttoptiebar{C}_{SG}) \cap X$. Now, by Claim \ref{closed path}, $u'\texttoptiebar{C}_{SG}v' \cup u'uvv'$ forms a closed $X$-3-path, a contradiction. Case (ii) $(N_G(u) \cap \texttoptiebar{C}_{SG}) \cap X \ne \emptyset$ and $(N_G(v) \cap \texttoptiebar{C}_{SG}) \cap X = \emptyset$. Since $u \not \in \texttoptiebar{C}_{SG}$, $(N_G(u) \cap \texttoptiebar{C}_{SG}) \subseteq X$. By Claim \ref{closed path}, there exists a closed $X$-3-path that contains some $p, q \in N_G(v) \cap \texttoptiebar{C}_{SG}$ in $\texttoptiebar{C}_{SG}$, say $R'$. Let $p' \in (N_G(p) \cap R') \cap X$, $p'' \in (N_G(p) \cap R') \cap (V \setminus X)$, $q' \in (N_G(q) \cap R') \cap X$, and $q'' \in (N_G(q) \cap R') \cap (V \setminus X)$. If $R' = pp' \cdots q'q \cup pp'' \cdots q''q$ then $pp''R'q''q \cup pvq$ forms a closed $X$-3-path, a contradiction. Thus $R' = pp' \cdots q''q \cup pp'' \cdots q'q$. Case (ii - 1) $p'$ and $q'$ have a $X$-3-path $S$ between them in $\texttoptiebar{C}_{SG}$ such that $(N_G(p') \cap S) \cap R' = \emptyset$ and $(N_G(q') \cap S) \cap R' = \emptyset$. Now, $q'R'p''p \cup pvq \cup qq''R'p' \cup p'Sq'$ forms a closed $X$-3-path, a contradiction. Case (ii - 2) $(N_G(p') \setminus R') \cap \texttoptiebar{C}_{SG} = \emptyset$ and $(N_G(q') \setminus R') \cap \texttoptiebar{C}_{SG} = \emptyset$. By Claim \ref{closed path}, without loss of generality, $N_G((N_G(p') \cap R') \setminus \{p\}) \cap Y \ne \emptyset$. It suffices that $v \in Y$. Now, $(Y \cup \{p\}) \setminus \{v, p'\}$ is a dominating set of $G$, that is contrary to the minimality of $Y$. Case (ii - 3) $(N_G(p') \setminus R') \cap X \ne \emptyset$ and $(N_G(q') \setminus R') \cap X \ne \emptyset$. By Claim \ref{closed path}, without loss of generality, $N_G((N_G(p') \cap R') \setminus \{p\}) \cap Y \ne \emptyset$. It suffices that $v \in Y$. Now, $(Y \cup \{p\}) \setminus \{v, p'\}$ is a dominating set of $G$, that is contrary to the minimality of $Y$. Case (ii - 4) $(N_G(p') \setminus R') \cap (\texttoptiebar{C}_{SG} \cap (V \setminus X)) \ne \emptyset$ and $(N_G(q') \setminus R') \cap \texttoptiebar{C}_{SG} = \emptyset$. Since $G$ is 3-connected, for a vertex $x \in N_G(q') \setminus R'$, we have $|N_G(x)| \geq 3$ and $N_G(x) \subseteq X$. Let $y \in N_G(x) \setminus \{q'\}$. For $y', y'' \in N_G(y)$, let $q''qq' \cdots y''yy' \cdots q''$ be a closed $X$-3-path in $\texttoptiebar{C}_{SG}$. It suffices that $v \in Y$. By Case (ii - 2), $Y$ is not a d-set of $G + vy''$ and $(Y \cup \{y''\}) \setminus \{v, y\}$ is a dominating set of $G + vy''$. Let $Y' = (Y \cup \{y''\}) \setminus \{v, y\}$. Now, $Y'$ is a d-set of $G + vy''$, $Y' \cup \{v\}$ is a dominating set of $G$, and so $Y' \cup \{v\}$ is a d-set of $G$. For some $D \in \mathcal{C}_t(G)$ with $\texttoptiebar{C}_{SG} \ne D$, $X'$-3-paths can be assigned to $D$ so that $v \in X'$ and $X' \subseteq Y' \cup \{v\}$ (see the proof of Lemma \ref{FG}). Now, $vy'' \not \in E(D)$ and $D \in \mathcal{C}_t(G + vy'')$. Since $Y'$ is a d-set of $G + vy''$, for some $D' \in \mathcal{C}_t(G + vy'')$ with $D \ne D'$, $X''$-3-paths can be assigned to $D'$ so that $X'' \subseteq Y'$. Now, $v \in X''$ and $vy'' \in E(D')$. Since $Y' \bigtriangleup (Y' \cup \{v\}) = \{v\}$, $D$ and $D'$ have a certain same closed $X'$-3-path ($X''$-3-path), and so $D = D'$, a contradiction.
 Case (ii - 5) $N_G(p') \setminus R' \subseteq X$. It suffices that $v \in Y$. Let $y \in N_G(p') \setminus R'$. For $y' \in N_G(y) \cap (V \setminus X)$, let $pp' \cdots y'y \cdots p$ be a closed $X$-3-path in $\texttoptiebar{C}_{SG}$. By Case (ii - 3), $Y$ is not a d-set of $G + vy'$ and $(Y \cup \{y'\}) \setminus \{v, y\}$ is a dominating set of $G + vy'$. Let $Y' = (Y \cup \{y'\}) \setminus \{v, y\}$. Now, $Y'$ is a d-set of $G + vy'$, $Y' \cup \{v\}$ is a dominating set of $G$, and so $Y' \cup \{v\}$ is a d-set of $G$. For some $D \in \mathcal{C}_t(G)$ with  $\texttoptiebar{C}_{SG} \ne D$, $X'$-3-paths can be assigned to $D$ so that $v \in X'$ and $X' \subseteq Y' \cup \{v\}$ (see the proof of Lemma \ref{FG}). Now, $vy' \not \in E(D)$ and $D \in \mathcal{C}_t(G + vy')$. Since $Y'$ is a d-set of $G + vy'$, for some $D' \in \mathcal{C}_t(G + vy')$ with $D \ne D'$, $X''$-3-paths can be assigned to $D'$ so that $X'' \subseteq Y'$. Now, $v \in X''$ and $vy' \in E(D')$. Since $Y' \bigtriangleup (Y' \cup \{v\}) = \{v\}$, $D$ and $D'$ have a certain same closed $X'$-3-path ($X''$-3-path), and so $D = D'$, a contradiction.
Case (iii) $(N_G(u) \cap \texttoptiebar{C}_{SG}) \cap X = \emptyset$ and $(N_G(v) \cap \texttoptiebar{C}_{SG}) \cap X = \emptyset$. We consider the graph $G - uv$. Let $Y'$ be a d-set of $G - uv$. Without loss of generality, $X \subseteq Y'$, $u, v \in Y'$, and $|Y'| = |Y| + 1$. By Case (ii), without loss of generality, $u$ and $v$ are contained in $\texttoptiebar{C}_{SG}$ of $G - uv$, therefore, $u$ and $v$ are contained in $\texttoptiebar{C}_{SG}$ of $G$, a contradiction.
Suppose $R = \{u\}$. Case (iv) $N_G(u) \cap X \ne \emptyset$ and $N_G(u) \cap (V \setminus X) \ne \emptyset$. Now, $u \in \texttoptiebar{C}_{SG}$, a contradiction. Case (v) $N_G(u) \subseteq V \setminus X$. Consider Case (ii). Therefore, $|R| \leq 1$ and if $|R| = 1$ then $u \in R$ satisfies $N_G(u) \subseteq X$.
\end{proof}

By Claim \ref{LC}, the proof of Theorem \ref{T2} is complete.

\end{proof}

\begin{cor}\label{cor}
Let $G$ be a 3-connected graph. For some $H \in \mathcal{F}(G)$, the d-set $X$ of $G$ is obtained by assigning $X$-3-paths to the component of $H$.
\end{cor}

\begin{proof}
It is straightforward from the proof of Theorem \ref{T2}. 
\end{proof}

We call the member $H \in \mathcal{F}(G)$ in Corollary \ref{cor} the domination structure of $G$.

\begin{lem}\label{PO}
For a 3-connected graph $G$, each $H \in \mathcal{F}(G)$ is determined in polynomial time.
\end{lem}

\begin{proof}
From the proof of Theorem \ref{T2}, the component of $H \in \mathcal{F}(G)$ is constructed by finding cycles of length 0 mod 3 one by one with one path intersection. For each step, these cycles increase at least one in number, and so $H$ is polynomially determined.
\end{proof}

\begin{lem}\label{FG}
For a 3-connected graph $G$, $|\mathcal{F}(G)| \leq |V(G)| + 1$.
\end{lem}

\begin{proof}
Let $H_1 \in \mathcal{F}(G)$. Let $X$-3-paths be assigned to the component of $H_1$. 
Case (A) $V(G) \ne V(H_1)$. Let $R$ be a component of $G - H_1$. Let $\mathcal{R}$ be the set of all $R$. For a vertex $x \in R \in \mathcal{R}$ and a vertex $c \in H_1$, let $E(x)$ be a set of $xc$ as the following. We construct $G'$ as $G + \bigcup_{x \in \bigcup_{R \in \mathcal{R}}R} E(x)$ so that $E(x)$ is maximal as long as the component of $H_1$ is contained and maximal also in $G'$. For a closed $X$-3-path $C$ of $H_1$, suppose that a member $R \in \mathcal{R}$ is such that $N_G(R) \cap C \ne \emptyset$. Since $G$ is 3-connected, we assume that $C$ is not a triangle. By the definition of $G'$, for a vertex $x' \in N_G(R) \cap C$, a vertex $x \in R \cap N_G(x')$ is adjacent to every three vertices of the $X$-3-path assigned to $C$ in $G'$. Since $C$ is not a triangle, $|N_{G'}(x) \cap C| \geq 2$, set $N_{G'}(x) \cap C \supseteq \{y_1, y_2\}$. Also, by the consideration of the length of a path between two vertices of $N_{G'}(R) \cap C$ through $R$ mod $3$, without loss of generality, we can assume $R = \{x\}$. Let $H_2 \in \mathcal{F}(G) \setminus \{H_1\}$. Let $W$-3-paths be assigned to the component of $H_2$. By the definition of $H_2$, a closed $W$-3-path $D$ such that $C \cap D \ne \emptyset$ is obtained by adding 2 or 3 $C$-paths to $C$. For two vertices $z_1, z_2 \in C \cap D$, suppose $A = z_1(C - D)y_1Ry_2(C - D)z_2Dz_1$, $B_1 = z_1(C - D)y_1Ry_2(C - D)z_2$, and $B_2 = z_2Dz_1$. If we can take $B_1$ as $|B_1| - |z_1Cz_2| \equiv 0$ mod $3$, then $|A| \equiv |D| \equiv 0$ mod $3$. Otherwise, without loss of generality, $y_1 = z_1$ and $y_2 = z_2$, and we can take $B_2$ as either $|B_2| - |z_1Cz_2| \equiv 1$ mod $3$ or $|B_2| - |z_1Cz_2| \equiv 2$ mod $3$, arbitrarily. Set $C = z_1C_1z_2C_2z_1$. 
Case (i) $|B_1| - |z_1C_2z_2| \equiv 2$ mod $3$. Now, for $|B_2| - |z_1C_1z_2| \equiv 1$ mod $3$, it follows that $|A| \equiv (|B_1| - |z_1C_2z_2|) + (|B_2| - |z_1C_1z_2|) \equiv 0$ mod $3$. Case (ii) $|B_1| - |z_1C_2z_2| \equiv 1$ mod $3$. Now, for $|B_2| - |z_1C_1z_2| \equiv 2$ mod $3$, it follows that $|A| \equiv (|B_1| - |z_1C_2z_2|) + (|B_2| - |z_1C_1z_2|) \equiv 0$ mod $3$. That is, the closed $W$-3-path that contains $R$ is obtained by adding one $D$-path, $B_1$, in $G'$. Thus, for some $H'_1 \in \mathcal{F}(G')$, $R$ is a component of $G' - H'_1$, and for all $H'_2 \in \mathcal{F}(G') \setminus \{H'_1\}$,  $R \subseteq H'_2$. Therefore, $|\mathcal{F}(G)| \leq |\mathcal{F}(G')| \leq |V(G')| = |V(G)|$.  
Case (B) $V(G) = V(H_1)$. For a new vertex $x \not \in V(G)$, let $R = \{x\}$ and $E(x) = \{xc\colon\ c \in X\}$. We construct $G'$ with the vertex set $V(G) \cup R$ and edge set $E(G) \cup E(x)$. By the same argument as Case (A), for some $H'_1 \in \mathcal{F}(G')$, $x \not \in H'_1$, and for all $H'_2 \in \mathcal{F}(G') \setminus \{H'_1\}$, $x \in H'_2$. Therefore, $|\mathcal{F}(G)| \leq |\mathcal{F}(G')| \leq |V(G')| = |V(G)| + 1$. 

\end{proof}

\begin{thm}
For a 3-connected graph $G$, the d-set of $G$ is determined in polynomial time.
\end{thm}

\begin{proof}
It is straightforward from Corollary \ref{cor}, Lemma \ref{PO}, and Lemma \ref{FG}.
\end{proof}

\newpage

\begin{appendices}

\section{}

The tables of $(*T1), (*T2)$ list the cases that $(*0)$ is applied to. Here, we represent some details for $(*0)$.
Consecutive three vertices in $P_1$, $P_2$, and $P_3$ are denoted by ($P_1^{0}$, $P_1^{1}$, $P_1^{2}$), ($P_2^{0}$, $P_2^{1}$, $P_2^{2}$), and ($P_3^{0}$, $P_3^{1}$, $P_3^{2}$) respectively. 
Set $\{P_1^{0}, P_1^{1}, P_1^{2}, P_2^{0}, P_2^{1}, P_2^{2}, P_3^{0}, P_3^{1}, P_3^{2}\} = \{1, 2, 3, 4, 5, 6, 7, 8, 9\}$. For $1 \leq i \leq 9$, $\mathcal{S}_i$ has 18 paths from 6 cases of opposite ends and 3 cases of lengths. \\

\noindent ($*T1$)
\begin{table}[h]
  \begin{minipage}[t]{.45\textwidth}
    \begin{center}
      \begin{tabular}{|c|c|c|c|c|c|c|c|c|c|c|} \hline
    \ & $(||P_1||, ||P_2||, ||P_3||)$ & $|\mathcal{T}_1|$ & $|\mathcal{T}_2|$ & $|\mathcal{T}_3|$ & $|\mathcal{T}_4|$ & $|\mathcal{T}_5|$ & $|\mathcal{T}_6|$ & $|\mathcal{T}_7|$ & $|\mathcal{T}_8|$ & $|\mathcal{T}_9|$ \\ \hline
(I) & (2, 0, 0) & 2 & 2 & & & & & & &  \\ \hline
(II) & (0, 0, 0) & 4 & 4 & 4 & 4 & & & & &  \\ \hline

      \end{tabular}
    \end{center}

  \end{minipage}
\end{table}

\newpage

\noindent ($*T2$)
\begin{table}[h]
  \begin{minipage}[t]{.45\textwidth}
    \begin{center}
      \begin{tabular}{|c|c|c|c|c|c|c|c|c|c|c|c|} \hline
   $(t, u, v)$ & $(||P_1||, ||P_2||, ||P_3||)$ & $|\mathcal{T}_1|$ & $|\mathcal{T}_2|$ & $|\mathcal{T}_3|$ & $|\mathcal{T}_4|$ & $|\mathcal{T}_5|$ & $|\mathcal{T}_6|$ & $|\mathcal{T}_7|$ & $|\mathcal{T}_8|$ & $|\mathcal{T}_9|$ \\ \hline
 ((a), (a), (a))  & (0, 0, 0) & 4 & 4 & 4 & 4 & & & & & \\ \hline
 ((a), (a), (a))  & (0, 0, 2) & 2 & 2 & & & & & & & \\ \hline

 ((a), (a), (b))  & (0, 0, 0) & 0 &  & & & & & & &  \\ \hline
 ((a), (a), (b))  & (0, 1, 1) & 0 &  & & & & & & &  \\ \hline
 ((a), (a), (b))  & (2, 1, 1) & 0 &  & & & & & & & \\ \hline

 ((a), (a), (c))  & (0, 1, 1) & 2 & 2 & & & & & & & \\ \hline
 ((a), (a), (c))  & (0, 1, 2) & 1 & 2 & & & & & & & \\ \hline
 ((a), (a), (c))  & (0, 2, 1) & 1 & 2 & & & & & & & \\ \hline

 ((a), (a), (d))  & (0, 1, 1) & 0 & & &  & & & & &  \\ \hline

 ((a), (b), (b))  & (0, 0, 0) & 0 &  &  &  & & & & &  \\ \hline

 ((a), (b), (c))  & none & &  &  &  & & & & &  \\ \hline

 ((a), (b), (d)) & none & &  &  &  & & & & &  \\ \hline

 ((a), (c), (c))  & (1, 0, 2) & 1 & 3 &  &  & & & & &  \\ \hline
 ((a), (c), (c))  & (2, 0, 1) & 1 & 3 &  &  & & & & &  \\ \hline
 ((a), (c), (c))  & (2, 0, 2) & 2 & 2 &  &  & & & & &  \\ \hline

 ((a), (c), (d)) & none & &  &  &  & & & & &  \\ \hline

 ((a), (d), (d))  & (1, 0, 1) & 0 & &  &  & & & & &  \\ \hline

 ((c), (b), (b))  & (2, 0, 2) & 0 &  &  &  & & & & &  \\ \hline

 ((c), (b), (c))  & (2, 2, 0) & 0 &  &  &  & & & & &  \\ \hline

 ((c), (b), (d))  & none & &  &  &  & & & & &  \\ \hline
 ((c), (c), (c))  & (1, 0, 0) & 2 & 3 & 4 &  & & & & &  \\ \hline
 ((c), (c), (d))  & (0, 0, 0) & 0 &  &  &  & & & & &  \\ \hline

 ((c), (d), (d)) & (0, 0, 0) & 0 &  &  &  & & & & &  \\ \hline

      \end{tabular}
    \end{center}

  \end{minipage}
\end{table}

\noindent Note that the choice of types (a), (b), (c), and (d) for the vertex $t$ is (a) or (c) in $(*T2)$. It suffices because other types are considered by the rotation of the vertices of $X$ in the same way.

\end{appendices}

\end{document}